\documentclass[a4paper,11pt]{amsproc}
\usepackage{amssymb}
\usepackage{amscd}
\usepackage[dvips]{graphicx}
\usepackage[all,cmtip]{xy}
\usepackage[inline]{enumitem}
\usepackage{xcolor}
\usepackage{float}
\usepackage{cite}
\usepackage[hidelinks]{hyperref}

\makeatletter
\renewcommand*{\eqref}[1]{%
  \hyperref[{#1}]{\textup{\tagform@{\ref*{#1}}}}%
}
\makeatother

\newtheorem{theorem}{Theorem}
\newtheorem{proposition}[theorem]{Proposition}
\newtheorem{lemma}[theorem]{Lemma}

\theoremstyle{definition}
\newtheorem{example}{Example}

\newtheorem{remark}{Remark}

\newcommand{\R}{\mathbb{R}}

\newcommand{\g}{\mathfrak{g}}

\newcommand{\e}{\mathbf{e}}

\newcommand{\uu}{\mathfrak{u}}

\newcommand{\so}{\mathfrak{so}}

\newcommand{\ad}{\mathrm{ad}}
\DeclareMathOperator{\Ad}{\mathrm{Ad}}
\DeclareMathOperator{\Span}{\mathrm{span\,}}
\DeclareMathOperator{\diag}{\mathrm{diag}}

\DeclareMathOperator{\rank}{\mathrm{rank}}
 \DeclareMathOperator{\pr}{\mathrm{pr}}
\DeclareMathOperator{\tr}{\mathrm{tr}}

\allowdisplaybreaks

\begin{document}

\title[Sub-Riemannian geodesics related to filtrations of Lie algebras]{Normal sub-Riemannian geodesics related to filtrations of Lie algebras}

\author{Bo\v zidar Jovanovi\'c, Tijana \v Sukilovi\' c, Srdjan Vukmirovi\'c}

\address{B.J.: Mathematical Institute, Serbian Academy of Sciences and
Arts, Kneza Mihaila 36, 11000 Belgrade, Serbia}
\email{bozaj@mi.sanu.ac.rs}

\address{ T.\v{S}, S.V:  Faculty of Mathematics, University of Belgrade, Studentski trg 16, 11000 Belgrade, Serbia}
\email{tijana.sukilovic@matf.bg.ac.rs, srdjan.vukmirovic@matf.bg.ac.rs}


\begin{abstract}
There is a natural way to construct sub-Riemannian structures that depend on $n$ parameters on compact Lie groups. These structures are related to the filtrations of Lie subalgebras $\g_0 < \g_1 < \g_2 < \dots < \g_{n-1}<\g_n=\g=Lie(G)$. In the case where $n=1$, the explicit solution for normal sub-Riemannian geodesics was provided by Agrachev,  Brockett, and Jurjdevic. We extend their solution to apply to general chains of Lie subgroups. Additionally, we describe normal geodesic lines of the induced sub-Riemannian structures on homogeneous spaces $G/K$, where $\mathfrak g_0=Lie(K)$.
\end{abstract}

\maketitle

\section{Introduction}

Consider the chain of compact Lie subgroups
\begin{align*}
G_0 < G_1 < G_2 < \dots < G_{n-1}\subset
G_n=G
\end{align*}
and the corresponding filtration of the Lie algebra $\g=Lie(G)$
\begin{align}\label{filtration}
\g_0 <\g_1<\g_2< \dots<\g_{n-1}< \g_n=\g.
\end{align}

We note that $"<"$ denotes the strict inclusion.
Fix an invariant scalar product
$\langle\,\cdot\,,\,\cdot\,\rangle$ on $\g$ and denote the
restrictions of $\langle\,\cdot\,,\,\cdot\,\rangle$ to $\g_i$
also by $\langle\,\cdot\,,\,\cdot\,\rangle$.
Let $\mathfrak p_i$ be the orthogonal complement of $\g_{i-1}$ in $\g_i$, and set $\g_0=\mathfrak p_0$.
Then $\g_i=\mathfrak p_0\oplus \mathfrak p_1 \oplus \dots\oplus \mathfrak p_i$.
For $\xi\in\g$, we set
\begin{align*}
\xi=\xi_0+\dots+\xi_n, \qquad \xi_{\mathfrak g_i}=\xi_0+\dots+\xi_i\in \g_i,  \qquad \xi_i=\pr_{\mathfrak p_i}\xi, \qquad i=0,\dots,n.
\end{align*}

Let us assume that for some set of indices
$\mathcal I \subsetneqq \{0,1,\dots,n\}$, the linear subspace
\begin{align}\label{defd}
\mathfrak d=\bigoplus_{i\in\mathcal I} \mathfrak p_i<\g
\end{align}
generate $\g$ by commutation. Then the corresponding left-invariant distribution $\mathcal D\subset TG$ is given by
\begin{align}\label{mathcal D}
\mathcal D\vert_g=dL_g(\mathfrak d), \qquad g\in G,
\end{align}
and it is completely nonholonomic. We consider normal geodesic lines
(see~\cite{ABB, Mo}) of the
left-invariant metrics $ds^2_{\mathcal D,s}$ defined by
the scalar product:
\begin{align}\label{sub-Rimanov proizvod}
(\xi,\eta)_\mathfrak d=\sum_{i\in\mathcal I} \frac{1}{s_i} \langle \xi_i,\eta_i\rangle, \quad \xi,\eta\in\mathfrak d,  \quad s_i>0, \quad i\in\mathcal I,
\end{align}

In the case where $\mathfrak d=\mathfrak p_1\oplus\dots\oplus \mathfrak p_n$, we recover the setting studied in~\cite{Jo, JSV2024}.
The corresponding normal sub-Riemannian geodesic flow on $T^*G$ is completely integrable in the non-commutative sense by means of integrals that are polynomial in momenta.
A generic motion is a quasi-periodic winding over $\Delta$--dimensional invariant isotropic tori~\cite{JSV2024}, where\footnote{Here we take
$x_i\in\mathfrak p_i$, $i=1,\dots,n$, such that the dimensions of
the isotropy algebras $\g_{i}(x_{\g_i})$ and $\g_{i-1}(x_{\g_i})$ are minimal.}
\begin{align}\label{Delta}
  \Delta=\rank\g_0+\sum_{i=1}^n
\dim\pr_{\mathfrak p_i}(\g_i(x_{\g_i})).
\end{align}

In particular, the simplest case arises when $\mathfrak d=\mathfrak p_1$ ($n=1$). In this setting, the sub-Riemannian metric $ds^2_{\mathcal D,s}$ is simply the restriction of the bi-invariant metric associated with $\langle\cdot,\cdot\rangle$, scaled by ${1}/{s_1}$, to the distribution $\mathcal D$.
This sub-Riemannian problem is commonly referred to as a $\mathfrak p_1\oplus\g_0$--problem on a Lie group $G$, and the
corresponding normal sub-Riemannian geodesics are given by:
\begin{align}\label{rekonstrukcija}
g(t)=\bar g\exp(ts_1 \bar x)\exp(-ts_1\bar x_0),
\end{align}
where $g(0)=\bar g$, and $\bar x=\bar x_0+\bar x_1\in\g$ is arbitrary
(see~\cite{Sac}). The expression~\eqref{rekonstrukcija} was independently derived by Agrachev~\cite{Ag} (for $\dim\g_0=1$) and by Brockett~\cite{Br} and Jurdjevic~\cite{Ju1999} (for a symmetric pair $(\mathfrak g,\mathfrak g_0)$).

By applying Souris' result~\cite{Su}, we obtain an explicit solution for normal sub-Riemannian geodesics on the space $(G,ds^2_{\mathcal D_s})$ (see Theorems~\ref{prva},~\ref{pomocna}). We also extend this solution to the corresponding sub-Riemannian structures on the homogeneous spaces $G/K$, $K=G_0$ (see Theorem~\ref{druga}). In particular, item~\ref{it:pomocna1} of Theorem~\ref{pomocna}
provides all solutions of
the Gel'fand--Cetlin systems on $\so(n)$ and $\uu(n)$, covering both regular and singular adjoint orbits~\cite{GS1, BMZ}.

The solutions stated in
Theorems~\ref{prva},~\ref{pomocna},  and~\ref{druga} are valid for all filtrations~\eqref{filtration}  without requiring the corresponding Lie subgroups $G_i$, with $\g_i=Lie(G_i)$, to be closed.
We only need that $K=G_0$ is a closed Lie subgroup of $G$ for the smoothness of the homogeneous space $G/K$. Also, we can consider an arbitrary semi-simple Lie group $G$ endowed with the Killing form $\langle\cdot,\cdot\rangle$. Along with the corresponding chains of Lie algebras~\eqref{filtration}, we impose the additional requirement that the restriction of the Killing form $\langle\cdot,\cdot\rangle$ to each subalgebra  $\g_i$ is non-degenerate.

Examples illustrating the construction and connection to sub-Riemannian structures with integrable normal geodesic flows—derived from the Manakov metrics on the orthogonal group  $SO(n)$ (Theorems~\ref{man1},~\ref{man2}) and a class of homogeneous spaces of $SO(n)$ (Theorem~\ref{man3})—are presented in Sections~\ref{sec4} and~\ref{sec5}. As a consequence, in Section~\ref{sec4}, we provide a positive answer for $G=SO(n)$ and $n-1\leq k <\frac{n(n-1)}{2}$ to the following natural question (see~\cite{BAB}): Is there a left-invariant sub-Riemannian structure of $\rank k$ with an integrable normal geodesic flow on any (semi-simple) Lie group $G$ and any admissible $k$?  Additionally, we present an example that demonstrates a negative answer for $k=2$ and $n>3$.

In~\cite{Jo}, we also investigated similar nonholonomic problems on Lie groups. Further examples of integrable sub-Riemannian geodesic flows on Lie groups and homogeneous spaces can be found, for instance, in~\cite{Sac, Ju, Ju2, BJ3, Po, PS}.

\section{The Agrachev--Brockett--Jurdjevic formula for chains of subalgebras}

\subsection{Bogoyavlensky's conjecture}

Using $\langle\,\cdot\,,\,\cdot\,\rangle$, we identify $\g\cong\g^*$ and
$\g_i\cong\g^*_i$, $i=0,\dots,n$.
In this subsection, we do not assume that $\mathfrak d$ is bracket generating subspace of~$\g$.
Recall that we use the
following notation
\begin{align}\label{oznake}
x=x_0+\dots+x_n, \quad x_{\mathfrak g_i}=x_0+\dots+x_i\in \g_i,  \quad x_i=\pr_{\mathfrak p_i}x, \quad i=0,\dots,n.
\end{align}

Consider the operator $A: \g^*\to \g$ and the quadratic left-invariant Hamiltonian function defined by
\begin{align*}
A(x)=A_0(x_0)+s_1x_1+\dots+s_nx_n, \qquad H(x)=\frac12\langle A(x),x\rangle,
\end{align*}
where $A_0:\g_0^*\to\g_0$ and $s_1,\dots,s_n$ are real parameters.

In left trivialization $T^*G\cong G\times \g^*(g,x)$, the corresponding Hamiltonian system is given~by
\begin{align}
\label{Euler}&\dot x=[x,\omega], \qquad \omega:=\nabla H\vert_x=A(x),\\
\label{KIN} &\dot g=d(L_g)(\omega)=d(L_g)(A_0(x_0)+s_1 x_1+\dots +s_n x_n).
\end{align}

The Euler equation~\eqref{Euler} forms a closed Hamiltonian system on $\g^*\cong\g$ equipped with the standard Lie–Poisson brackets.
The variable $\omega\in\g$ is commonly referred to as the angular velocity, while the variable $x\in\g^*\cong \g$ is known as the angular momentum (see~\cite{Ar}).

Due to the relations
\begin{align}\label{uslovi}
[\mathfrak p_i,\mathfrak p_j]\subset \mathfrak p_j,\qquad j=1,\dots,n, \qquad i=0,\dots,j-1,
\end{align}
Euler equation~\eqref{Euler} can be decomposed into the following form (see~\cite{B1, JSV}):
\begin{align}
&\dot x_0 = [x_0,A_0(x_0)], \label{Euler-0}\\
&\dot x_i = [s_i x_0-A_0(x_0)+(s_i-s_1) x_1+\dots+(s_i-s_{i-1})x_{i-1},x_i], \quad i=1,\dots,n. \label{Euler-i}
\end{align}

In~\cite{B1} Bogoyavlensky conjectured that if the Euler equation~\eqref{Euler-0} is completely integrable, then the extension of the system~\eqref{Euler-0},~\eqref{Euler-i} is also completely integrable.
The simplest examples are the Gel'fand--Cetlin integrable systems related to natural filtrations of Lie algebras $\so(n)$ and $\uu(n)$ and $\g_{0}=\{0\}$~\cite{GS1, BMZ}.
In~\cite{JSV} we proved the conjecture: if the Euler equation~\eqref{Euler-0} is completely integrable, then the system~\eqref{Euler-0},~\eqref{Euler-i} is completely integrable in the non-commutative sense~\cite{MF, N}. We also
provide examples where there exist complete sets of additional Lie-Poisson commutative polynomial integrals~\cite{JSV, JSV2024}.
In the important case, when $(\g_i,\g_{i-1})$ are symmetric pairs, the commuting polynomial integrals are obtained by Mikityuk~\cite{Mik}.

Note that by the Noether theorem, the components of the momentum mapping
$\Phi(g,x)=\Ad_g(x)$
of the left $G$--action are always integrals of the system~\eqref{Euler-0},~\eqref{Euler-i},~\eqref{KIN},
implying non-commutative integrability
on the total phase space $T^*G$ (see~\cite{MF}).

\subsection{Sub-Riemannian geodesic flows related to filtration of Lie algebras}\label{sec4}

Let us return to the sub-Riemannian problem $(G,ds^2_{\mathcal D_s})$.
The Hamiltonian function of the normal sub-Riemannian geodesic flow,
\begin{align*}
H_{sR}(g,x)=H_{sR}(x)=\frac{s_0}2 \langle x_0,x_0\rangle+\dots+\frac{s_n}2\langle x_n,x_n\rangle, \quad s_i=0, \quad i\notin \mathcal I,
\end{align*}
belong to the class of problems studied by Bogoyavlensky  with $A_0(x_0)\equiv s_0 x_0$.
The corresponding normal sub-Riemannian geodesic flow is given by
\begin{align}
&\dot x_0 = 0, \label{Euler-0*}\\
&\dot x_i = [s_i x_0+(s_i-s_1) x_1+\dots+(s_i-s_{i-1})x_{i-1},x_i], \quad i=1,\dots,n, \label{Euler-i*}\\
&\dot g= d(L_g)(s_1 x_1+\dots +s_n x_n). \label{KIN*}
\end{align}

We assume that
\begin{align*}
s_i\ne s_{i-1},\qquad i=1,2,\dots,n.
\end{align*}
Otherwise, the problem can be reduced to a chain with fewer Lie subalgebras.

We have the following statement describing normal sub-Riemannian geodesics on $(G,ds^2_{\mathcal D,s})$.

\begin{theorem}\label{prva}
\begin{enumerate}[wide,label=(\roman*)]
    \item\label{it:prva1} The normal sub-Riemannian geodesic flow~\eqref{Euler-0*},~\eqref{Euler-i*},~\eqref{KIN*}  is completely integrable in a non-commutative sense by means of integrals polynomial in momenta. A generic motion is a quasi-periodic winding over $\Delta$--dimensional invariant isotropic tori within $T^*G$, where $\Delta$ is given by~\eqref{Delta}.

     \item\label{it:prva2} The normal sub-Riemannian geodesics of the metric $ds^2_{\mathcal D,s}$ are given by
\begin{align*}
g(t)=\bar g\exp(ts_n\bar x)\exp(t(s_{n-1}-s_n)\bar x_{\g_{n-1}})\cdots \exp(t(s_1-s_2)\bar x_{\g_1})\exp(t(s_0-s_1)\bar x_0),
\end{align*}
where $g(0)=\bar g$, $\bar x\in\g^*\cong\g$ is arbitrary, and $\bar x_{\g_i}=\bar x_0+\dots+\bar x_i$.
\end{enumerate}
\end{theorem}

\begin{example}
In the basic case where $\mathfrak d=\mathfrak p_1\oplus\dots\oplus \mathfrak p_n$ ($s_0=0, s_i>0, i\ge 1$), as a direct generalization of the Agrachev--Brockett--Jurjdevic solution~\eqref{rekonstrukcija},
the normal geodesic lines are of the form
\begin{align*}
g(t)=\bar g\exp(ts_n\bar x)\exp(t(s_{n-1}-s_n)\bar x_{\g_{n-1}})\cdots \exp(t(s_1-s_2)\bar x_{\g_1})\exp(-ts_1)\bar x_0).
\end{align*}
\end{example}

\begin{proof}
\ref{it:prva1} The proof of the first statement is the same as the proof for the case $\mathcal I=\{1,\dots,n\}$ given in~\cite{JSV2024}.

\ref{it:prva2} Consider the left-invariant Riemannian metrics $ds^2_s$ defined by the scalar product
\begin{align}\label{rimanova}
(\xi,\eta)_I=\langle I(\xi), \eta\rangle  \qquad \xi,\eta\in\mathfrak g,
\end{align}
where the operator $I$ (the inertia tensor in the case of rigid bodies~\cite{Ar}) is given by
\begin{align*}
I(\xi)= \hat s^{-1}_0\xi_0+\dots+\hat s^{-1}_n\xi_n, \qquad \xi\in\g, \qquad \hat s_i=s_i, \qquad i\in\mathcal I.
\end{align*}

In~\cite{Su}, a remarkable explicit solution for the geodesic lines on the Lie group $G$ for the metrics $ds^2_s$ is given for all decompositions
$\g=\mathfrak p_0+\mathfrak p_1+\dots+\mathfrak p_n$ satisfying~\eqref{uslovi}. More precisely, Souris considered geodesics on homogeneous spaces $G/H$, but we set $H$ to be the neutral of the group $G$ in~\cite[Theorem 1.1]{Su}.
In our notation, the geodesic line with the initial position $g(0)=\bar g$, and the initial angular velocity
\begin{align*}
\omega(0)=(dL_{g})^{-1}\dot g\vert_{t=0}=\bar\omega
\end{align*}
is given by
\begin{align}\label{resenje}
g(t)=\bar g\exp(t\hat s_n \bar x)\exp(t(\hat s_{n-1}-\hat s_n)\bar x_{\g_{n-1}})\cdots \exp(t(\hat s_1-\hat s_2)\bar x_{\g_1})\exp(t(\hat s_0-\hat s_1)\bar x_0),
\end{align}
where the initial angular momentum $\bar x$ is
\begin{align*}
\bar x=I(\bar\omega)=\hat s^{-1}_0\bar\omega_0+\dots+\hat s^{-1}_n\bar\omega_n.
\end{align*}

From the perspective of dynamics, Souris considered the Lagrangian formulation of the problem on the tangent bundle $TG\cong G\times \g(g,\omega)$:
\begin{align}\label{lag}
\frac{d}{dt}\big(I(\omega)\big)=[I(\omega),\omega], \qquad \dot g=d(L_g)(\omega), \qquad  I(\omega)=\nabla L\vert_\omega,
\end{align}
where the left-invariant Lagrangian is given by $L(\omega)=\frac12\langle I\omega,\omega\rangle$.
In the Hamiltonian formulation,  the geodesic flow of the metric $ds^2_s$ is expressed as
\begin{align}\label{ham}
\dot x=[x,\omega], \qquad \dot g=d(L_g)(\omega), \qquad \omega=I^{-1}(x)=\nabla H\vert_x,
\end{align}
where the Hamiltonian $H=\frac12\langle I^{-1}x,x\rangle$ is the Legendre transformation of $L$.

Again, the system~\eqref{ham} falls within the class of problems studied by Bogoyavlensky, where $A_0(x_0)=\hat s_0 x_0$.
The Riemannian metrics $ds^2_s$ tame the sub-Riemannian metric $ds^2_{\mathcal D,s}$: the restriction of the metrics
$ds^2_s$ to the distribution $\mathcal D$ coincide with $ds^2_{\mathcal D,s}$ (e.g, see~\cite[page 32]{Mo}).
In the limit
\begin{align*}
\hat{s}_i\to 0, \qquad i\notin\mathcal I,
\end{align*}
the operator $I$ is singular, but the corresponding Hamiltonian $H$ becomes the Hamiltonian $H_{sR}$ of the sub-Riemannian metric $ds^2_{\mathcal D,s}$. Consequently, the Hamiltonian equations~\eqref{ham} reduce to the Hamiltonian equations of the normal sub-Riemannian geodesic flow given by~\eqref{Euler-0*},~\eqref{Euler-i*}, and~\eqref{KIN*}.
Moreover, the curve~\eqref{resenje} is well defined for $\hat s_i=0$, $i\notin\mathcal I$, and it represents the projection to $G$ of the solution $(g(t),x(t))$
of the system~\eqref{Euler-0*},~\eqref{Euler-i*},~\eqref{KIN*} with the initial conditions $g(0)=\bar g$, $x(0)=\bar x$ (see Theorem~\ref{pomocna}).
\end{proof}

Following Bogoyavlensky’s approach~\cite{B1}, instead of focusing solely on a compact connected Lie group $G$ equipped with an invariant scalar product $\langle\cdot,\cdot\rangle$,
we can also consider an arbitrary semisimple Lie group $G$ endowed with the Killing form $\langle\cdot,\cdot\rangle$. Along with the corresponding chains of Lie algebras~\eqref{filtration}, we impose the additional requirement that the restriction of the Killing form $\langle\cdot,\cdot\rangle$ to each subalgebra  $\g_i$ is non-degenerate.
We identify the dual spaces $\g_i\cong \g_i^*$ using the (generally pseudo-Euclidean) scalar product $\langle\cdot,\cdot\rangle$.

Assume that $G$ is a connected compact or a semi-simple Lie group with the associated chain of subalgebras~\eqref{filtration}, and with the notation introduced above. 

With the notation~\eqref{oznake}, consider a left-invariant quadratic Hamiltonian function of the form
\begin{align*}
H(x)=\frac{s_0}2 \langle x_0,x_0\rangle+\frac{s_1}2 \langle x_1,x_1\rangle+\dots+\frac{s_n}2\langle x_n,x_n\rangle,
\end{align*}
for arbitrary real values of the parameters $s_0,s_1,\dots,s_n$. The Hamiltonian equations on $T^*G$ are:
\begin{align}
&\dot x_0 = 0, \label{Euler-0**}\\
&\dot x_i = [(s_i-s_0) x_0+(s_i-s_1) x_1+\dots+(s_i-s_{i-1})x_{i-1},x_i], \quad i=1,\dots,n, \label{Euler-i**}\\
&\dot g= d(L_g)(s_0 x_0+s_1 x_1+\dots +s_n x_n). \label{KIN**}
\end{align}

\begin{theorem}\label{pomocna}
\begin{enumerate}[wide,label=(\roman*)]
    \item\label{it:pomocna1} The solution of the Euler equations~\eqref{Euler-0**},~\eqref{Euler-i**}, with the initial condition  $x(0)=\bar x$,
is given by
\begin{align}\label{res}
 x_0(t)&=\bar x_0=\bar x_{\g_0}, \notag\\
 x_1(t)&=\Ad_{\exp(t(s_1-s_0)\bar x_{\g_0})}(\bar x_1),\notag\\
 x_2(t)&=\Ad_{\exp(t(s_1-s_0)\bar x_{\g_0})}\circ \Ad_{\exp(t(s_2-s_1)\bar x_{\g_1})} (\bar x_2), \\
 &\dots\notag\\
 x_n(t)&=\Ad_{\exp(t(s_1-s_0)\bar x_{\g_0})}\circ \Ad_{\exp(t(s_2-s_1)\bar x_{\g_1})}\circ\dots\circ \Ad_{\exp(t(s_n-s_{n-1})\bar x_{\g_{n-1}})} (\bar x_n).\notag
\end{align}
\item\label{it:pomocna2} The solution of the corresponding reconstruction problem on the Lie group~\eqref{KIN**}, with the initial condition $g(0)=\bar g$, is
\begin{align*}
g(t)=\bar g\exp(ts_n \bar x)\exp(t(s_{n-1}-s_n)\bar x_{\g_{n-1}})\cdots \exp(t(s_1-s_2)\bar x_{\g_1})\exp(t(s_0-s_1)\bar x_{\g_0}).
\end{align*}
\end{enumerate}
\end{theorem}

{\begin{proof}
    Define $v_k=(s_{k+1}-s_k)\bar x_{\g_k}$ and introduce the notation
    \begin{align*}
        A_{[i,j)}^t:=\Ad_{\exp(t v_i)}\circ \Ad_{\exp(t v_{i+1})}\circ\ldots\circ \Ad_{\exp(t v_{j-1})}, \quad 0\le i\le j\le n,
    \end{align*}
with the convention $A^t_{[i,i)}:=\mathrm{Id}$.
Observe that
\begin{align*}
    A^t_{[i,i+1)}(\bar x_{\g_i})=\bar x_{\g_i},\quad A^t_{[i,j)}\circ A^t_{[j,k)}=A^t_{[i,k)}\quad\text{and}\quad A^t_{[i,j)}([x, y])=[A^t_{[i,j)}(x), A^t_{[i,j)} (y)].
\end{align*}
With this notation, equations~\eqref{res} take the form
\begin{align*}
   x_0(t)&=\bar x_0=\bar x_{\g_0},\\
   x_i(t)&=A^t_{[0,i)}(\bar x_i),\quad i=1,\ldots,n.
\end{align*}
We now check that $x_i(t)$ indeed satisfies equation~\eqref{Euler-i**}.
\begin{align*}
    \frac{d}{dt}x_i(t)&=\frac{d}{dt}A^t_{[0,i)}(\bar x_i)=\frac{d}{dt}\left(A^t_{[0,1)}\circ A^t_{[1,i)}(\bar x_i)\right)=\frac{d}{dt}\left(\exp(t v_0) A^t_{[1,i)}(\bar x_i)\exp(-t v_0)\right)\\
    &=A^t_{[0,1)}[v_0, A^t_{[1,i)}(\bar x_i)]+A^t_{[0,1)}\left(\frac{d}{dt}A^t_{[1,i)}(\bar x_i)\right)\\
    &=[v_0, A^t_{[0,i)}(\bar x_i)]+A^t_{[0,1)}\left(\frac{d}{dt}\left(A^t_{[1,2)}\circ A^t_{[2,i)}(\bar x_i)\right)\right)\\
    &=[v_0, A^t_{[0,i)}(\bar x_i)]+A^t_{[0,1)}\left([v_1, A^t_{[1,i)}(\bar x_i)]+A^t_{[1,2)}\left(\frac{d}{dt}A^t_{[2,i)}(\bar x_i)\right)\right)\\
    &=[v_0, A^t_{[0,i)}(\bar x_i)]+[A^t_{[0,1)} v_1, A^t_{[0,i)}(\bar x_i)]+A^t_{[0,2)}\left(\frac{d}{dt}A^t_{[2,i)}(\bar x_i)\right)\\
    &=\ldots=[v_0+A^t_{[0,1)} v_1+\ldots A^t_{[0,i-1)} v_{i-1}, A^t_{[0,i)}(\bar x_i)]=[\sum_{k=0}^{i-1} A^t_{[0,k)} v_k, x_i]\\
    &=[\sum_{k=0}^{i-1} (s_{k+1}-s_k)A^t_{[0,k)}\bar x_{\g_{k}}, x_i]=[\sum_{k=0}^{i-1} (s_{k+1}-s_k)A^t_{[0,k)}(\bar x_{\g_{k-1}}+\bar x_k), x_i]\\
    &=[\sum_{k=0}^{i-1} (s_{k+1}-s_k)(A^t_{[0,k)}\bar x_{\g_{k-1}}+A^t_{[0,k)}\bar  x_k), x_i]\\
    &=[\sum_{k=0}^{i-1} (s_{k+1}-s_k)(A^t_{[0,k-1)}(A^t_{[k-1,k)}\bar x_{\g_{k-1}})+ x_k), x_i]\\
    &=[\sum_{k=0}^{i-1} (s_{k+1}-s_k)(A^t_{[0,k-1)}\bar x_{\g_{k-1}}+ x_k), x_i]\\
    &=[\sum_{k=0}^{i-1} (s_{k+1}-s_k)(A^t_{[0,k-2)}\bar x_{\g_{k-2}}+x_{k-1}+ x_k), x_i]\\
    &=\ldots=[\sum_{k=0}^{i-1} (s_{k+1}-s_k)\sum_{j=0}^k x_j, x_i]=[\sum_{j=0}^{i-1} \sum_{k=j}^{i-1} (s_{k+1}-s_k)x_j, x_i]\\
    &=[\sum_{j=0}^{i-1} (s_{i}-s_j)x_j, x_i].
\end{align*}
This concludes the proof of~\ref{it:pomocna1}. Since the proof of statement~\ref{it:pomocna2} is dual to that of~\cite[Theorem 1.1]{Su}, it will be omitted.
\end{proof}
}

\begin{remark}
The statement is dual to Souris's solution of the system~\eqref{lag} in the Lagrangian formulation~\cite{Su}, which allows us to handle arbitrary values of the parameters $s_0,\dots,s_n$.
When some of these parameters are equal to zero, the Hamiltonian
$H(x)$ becomes singular, and consequently, the Lagrangian (obtained via the Legendre transform of $H$) ceases to exist.
We presented the proof of item~\ref{it:pomocna1} of Theorem~\ref{pomocna},
since it provides all solutions of
the Gel'fand--Cetlin systems on $\so(n)$ and $\uu(n)$, covering both regular and singular adjoint orbits~\cite{GS1, BMZ}.
\end{remark}

\begin{example}
If $G$ is compact, $\dim\g_0=1$, $s_0<0$, $s_1,\dots,s_n>0$, $H(x)$ represents the Hamiltonian function for the Lorentz
metric on $G$ with completely integrable geodesic flow and geodesics of the form~\eqref{resenje}.
\end{example}

\section{Reduction to homogeneous spaces}\label{redukcija}

In this section, we assume that $\mathcal I \subsetneqq \{1,\dots,n\}$. We consider the associated sub-Riemannian problem on the homogeneous space $G/K$, where $K=G_0$,
from the perspective of symplectic reduction.

As above, assume that
\begin{align*}
\mathfrak d=\bigoplus_{i\in\mathcal I} \mathfrak p_i
\end{align*}
generate $\g$ by commutation and consider the Hamiltonian
\begin{align*}
H_{sR}(g,x)=H_{sR}(x)=\frac{s_1}2\langle x_1,x_1\rangle+\dots+\frac{s_n}2\langle x_n,x_n\rangle, \quad s_i=0, i\notin \mathcal I.
\end{align*}
of the normal sub-Riemannian geodesic flow on $T^*G$.

Recall that the natural right $K$--action on $T^*G$ is Hamiltonian. The zero value level-set of the momentum map is given by
\begin{align*}
(T^*G)_0=\{(g,x)\in T^*G\cong G\times\g^*\, \vert \, \pr_{\mathfrak k}(x)=0\}, \qquad \mathfrak k=\g_0,
\end{align*}
and the Marsden-Weinstein reduced space $(T^*G)_0/K$ is diffeomorphic to the cotangent bundle of the homogeneous space $G/K$ endowed with the standard symplectic structure (see~\cite{MMR}).
Let
$\pi\colon G\to G/K$
be the natural projection and $o:=\pi(e)=K$ be the projection of the neutral $e$ of $G$. The tangent space $T_o(G/K)$ is naturally identified with the orthogonal complement of $\mathfrak k$ within $\g$ with respect to $\langle\cdot,\cdot\rangle$:
\begin{align*}
T_o(G/K)\cong \mathfrak p_1\oplus \mathfrak p_2\oplus \dots \oplus \mathfrak p_n.
\end{align*}

The Hamiltonian function $H_{sR}$ is both left $G$--invariant and right $K$--invariant.  The normal sub-Riemannian geodesic flow on $T^*G$ induces well defined Hamiltonian system on the reduced space $(T^*G)_0/K\cong T^*(G/K)$ with the Hamiltonian function $H_{sR,0}$ obtained from the restriction of $H_{sR}$ to $(T^*G)_0$.

The reduced system corresponds to the normal sub-Riemannian geodesic flow on the $G$--invariant bracket generating distribution $\mathcal D_0$,
\begin{align*}
\mathcal D_0\vert_o \cong \mathfrak d=\bigoplus_{i\in\mathcal I} \mathfrak p_i,
\end{align*}
equipped with the $G$--invariant sub-Riemannian structure $ds^2_{\mathcal D_0,s}$ induced by the scalar product~\eqref{sub-Rimanov proizvod}.
Thus, the normal sub-Riemannian geodesics $\gamma(t)$ on $(G/K,ds^2_{\mathcal D_0,s})$ are projections of the solutions $(g(t),x(t))$ of the normal sub-Riemannian geodesic flow on $T^*G$ that belong to the invariant subspace $(T^*G)_0$:
\begin{align*}
\gamma(t)=\pi(g(t))=\pi\big(\bar g\exp(ts_n\bar x)\exp(t(s_{n-1}-s_n)\bar x_{\g_{n-1}})\cdots \exp(t(s_1-s_2)\bar x_{\g_1})\big),
\end{align*}
where $\pr_{\mathfrak k} \bar x=0$.
We summarize the above discussion in the following statement.

\begin{theorem}\label{druga}
The sub-Riemannian geodesics on $(G/K,ds^2_{\mathcal D_0,s})$ staring at the origin $\gamma(0)=o$  are of the form
\begin{align}\label{homGeo}
\gamma(t)=\exp(ts_n\bar x)\cdot \exp(t(s_{n-1}-s_n)\bar x_{\g_{n-1}})\cdots \exp(t(s_1-s_2)\bar x_{\g_1})\cdot o,
\end{align}
where $\pr_{\mathfrak k} \bar x=0$.
\end{theorem}

\begin{remark}
The solutions stated in
Theorems~\ref{prva},~\ref{pomocna},  and~\ref{druga} are valid
for all filtrations~\eqref{filtration} without requiring the corresponding Lie subgroups $G_i$, with $\g_i=Lie(G_i)$, to be closed.
The only essential assumption is that $K=G_0$ is a closed Lie subgroup of $G$.
\end{remark}

The normal sub-Riemannian geodesics~\eqref{homGeo} arise as limits of the Riemannian geodesics on $G/K$
endowed with the submersion metrics induced from $(G,ds^2_s)$, where $ds^2_s$ is left $G$-invariant and right $K$-invariant Riemannian metric defined by~\eqref{rimanova} (see~\cite{Su}).

The commutative integrability
of geodesic flows on homogeneous spaces, by means of integrals polynomial in momenta related to filtrations
$\g_0<\g_1<\dots<\g_n$,
has been studied in~\cite{Th, BJ1, BJ3}.

\section{Examples}\label{sec4}

\subsection{$SU(3)<G_2<SO(7)$}
  Let us consider the chain of algebras $\mathfrak{su}(3)< \mathfrak{g}_2<\mathfrak{so}(7)$. Denote the standard basis of the Lie algebra $\mathfrak{so}(7)$ by $\{\e_{ij}=\e_i\wedge \e_j\mid 1 \leq i < j \leq 7\}$. With respect to this basis, the vectors
  \begin{align*}
    P_0&=\e_{32}+\e_{67}, &
    P_1&=\e_{13}+\e_{57}, &
    P_2&=\e_{21}+\e_{74}, &
    P_3&=\e_{14}+\e_{72},\\
    P_4&=\e_{51}+\e_{37}, &
    P_5&=\e_{35}+\e_{17}, &
    P_6&=\e_{43}+\e_{61}, &&\\
    Q_0&=\e_{45}+\e_{67},&
    Q_1&=\e_{64}+\e_{57}, &
    Q_2&=\e_{65}+\e_{74},&
    Q_3&=\e_{36}+\e_{72},\\
    Q_4&=\e_{26}+\e_{37},&
    Q_5&=\e_{35}+\e_{42},&
    Q_6&=\e_{43}+\e_{52},&&
  \end{align*}
  form a basis of the exceptional Lie algebra $\mathfrak{g}_2$, while the vectors $P_0, Q_0,\ldots,Q_6$ span the algebra $\mathfrak{su}(3)$. Then:
  \begin{align*}
      \mathfrak p_0 &=\mathfrak{su}(3)=\Span\{P_0,Q_0,\ldots,Q_6\},\qquad \mathfrak p_1=(\mathfrak{su}(3))^\perp=\Span\{P_1,\ldots,P_6\},\\
      \mathfrak p_2&=(\mathfrak{g}_2)^\perp=\Span\{R_0,\ldots,R_6\}\\
      &\phantom{=(\mathfrak{g}_2)^\perp(}=\Span\{ \e_{71}+\e_{24}+\e_{35},
      \e_{16}+\e_{25}+\e_{43}, \e_{51}+\e_{26}+\e_{73},\\
      &\phantom{=(\mathfrak{g}_2)^\perp=\Span((}
      \e_{14}+\e_{27}+\e_{36}, \e_{32}+\e_{45}+\e_{76}, \e_{31}+\e_{46}+\e_{57}, \e_{21}+\e_{47}+\e_{65}\}.
  \end{align*}

Hence, we can take $\mathfrak{d}=\mathfrak p_2$ ($s_0=s_1=0, s_2\ne 0$),
$\mathfrak d=\mathfrak p_1\oplus \mathfrak p_2$ ($s_0=0, s_1,s_2\ne 0$), or $\mathfrak d=\mathfrak p_0\oplus\mathfrak p_2$ ($s_1=0, s_0,s_2\ne 0$) and the corresponding left-invariant sub-Riemannian metrics $ds^2_{\mathcal{D},s}$ with
normal sub-Riemannian geodesic lines described in Theorem~\ref{prva}.

Further, let us consider the homogeneous space $SO(7)/SU(3)$. We have the homogeneous fibration
\begin{align*}
G_2/SU(3)\cong S^6 \longrightarrow
SO(7)/SU(3) \longrightarrow
SO(7)/G_2\cong {\mathbb RP}^7
\end{align*}
with the horizontal space $\mathcal D_0$
induced from $\mathfrak d=\mathfrak p_2$
($s_0=s_1=0, s_2>0$).
Theorem~\ref{druga} provides the
normal geodesic lines for the
sub-Riemannian metric $ds^2_{\mathcal D_0,s}$.

\subsection{$U(1)<SU(2)<U(2)<SO(4)$}\label{ssec:42}
Consider the chain
of subalgebras
\begin{align*}
\g_0=\uu(1)<\g_1=\mathfrak{su}(2)<\g_2=\uu(2)<\g_3=\so(4),
\end{align*}
given by
\begin{align*}
&\uu(2)=\Span\{\e_{12}, \e_{34}, \e_{14}-\e_{23},\e_{13}+\e_{24}\}, \\
&\mathfrak{su}(2)=\Span\{\e_{12} - \e_{34}, \e_{14}-\e_{23},\e_{13}+\e_{24}\},\\
&\uu(1)=\Span\{\e_{12} - \e_{34}\},
\end{align*}
and the orthogonal decomposition
\begin{align*}
&\so(4)=\uu(1)\oplus\mathfrak{p_1}\oplus\mathfrak p_2\oplus \mathfrak p_3,
\quad\,\,\,\mathfrak p_1=\Span\{\e_{14}-\e_{23},\e_{13}+\e_{24}\}, \\
&\mathfrak p_2=\Span\{\e_{12}+\e_{34}\},\qquad \qquad \mathfrak p_3=\Span\{\e_{14}+\e_{23},\e_{13}-\e_{24}\}.
\end{align*}

A generic Hamiltonian related to the filtration is
\begin{align*}
H=\frac{1}{4}&\left({s_0}(x_{12}-x_{34})^2+{s_1}\big((x_{14}-x_{23})^2+(x_{13}+x_{24})^2\big)\right.\\
&\left.+
{s_2}(x_{12}+x_{34})^2+{s_3}\big((x_{14}+x_{23})^2+(x_{13}-x_{24})^2\big)\right),\quad s_0,s_1,s_2,s_3\geq 0.
\end{align*}
Thus, we obtain a two-parameter family of sub-Riemannian structures on the distribution $\mathcal D$ associated with:
\begin{align*}
\mathfrak d=\mathfrak p_1\oplus\mathfrak p_3=\Span\{\e_{13},\e_{14},\e_{23},\e_{24}\} \qquad (s_0=0, s_2=0, s_1,s_3>0).
\end{align*}

Note that, by employing the symmetric pair $(\so(4), \so(2) \oplus \so(2))$, we obtain a one-parameter family of sub-Riemannian structures on $\mathcal D$, which correspond to the restrictions of bi-invariant metrics.
Also note that the decomposition
$\so(4)=\mathfrak{su}(2)\oplus(\mathfrak p_2\oplus \mathfrak p_3)$ coincides with a well known decomposition of
$\so(4)$ into direct sum of two copies of $\so(3)$:
$\so(4)\cong \so(3)\oplus\so(3)$.

Additionally, we obtain a three-parameter family of sub-Riemannian structures on the codimension-1 distribution $\mathcal D$, related to:
\begin{align*}
\mathfrak d=\mathfrak p_1\oplus\mathfrak p_{2}\oplus\mathfrak p_3 \qquad (s_0=0, s_1, s_2,s_3>0).
\end{align*}
Solution curves of the Euler equation for the distributions corresponding to $\mathfrak d=\mathfrak p_1\oplus\mathfrak p_3$ are displayed in Figure~\ref{fig:example2} (\emph{left}), while those for $\mathfrak d=\mathfrak p_1\oplus\mathfrak p_{2}\oplus\mathfrak p_3$ appear in Figure~\ref{fig:example2} (\emph{right}).

Moreover, by applying the standard filtrations
\begin{align*}
& \g_0=\so(2) <\g_1=\so(3) < \g_2=\so(4), \\
& \g_0=\so(2)<\g_1=\so(2)\oplus\so(2)<\g_2=\so(4),
\end{align*}
and $\mathfrak d=\so(2)^\perp$, we obtain two 2-parameter families of sub-Riemannian structures on the corresponding codimension-1 distribution on $SO(4)$.

\begin{figure}[h]
    \centering
    \includegraphics[width=0.43\linewidth]{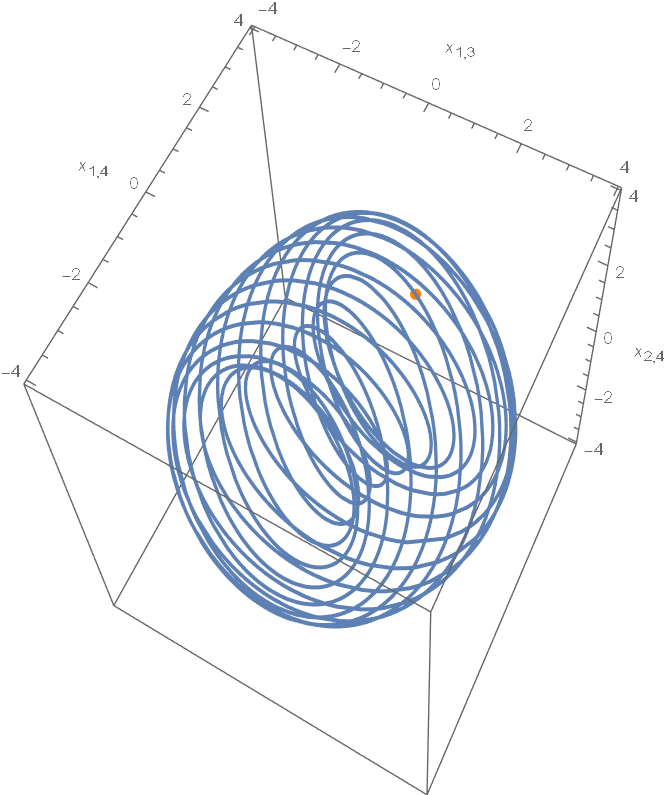}
    \includegraphics[width=0.43\linewidth]{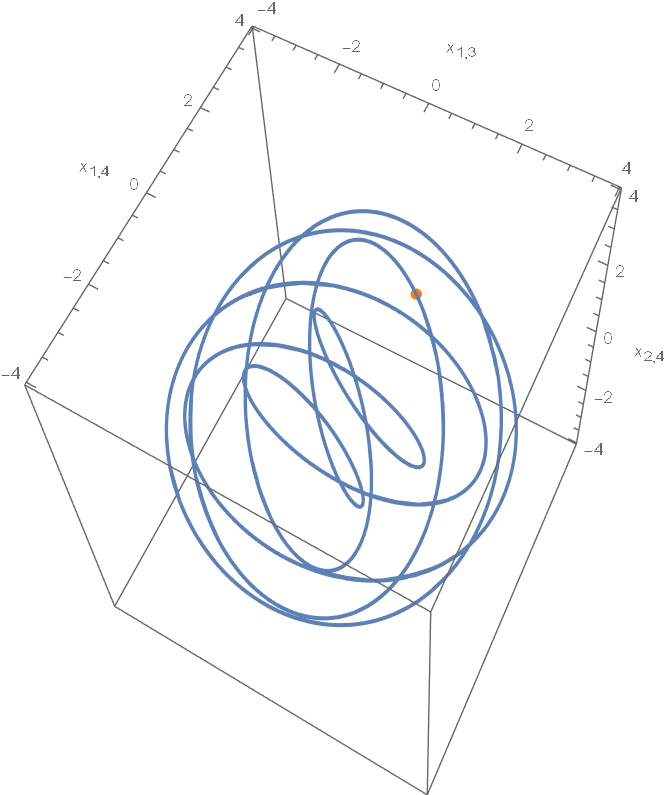}
    \caption{Projection of the solution curves of the Euler equation on the subspace $\Span\{\e_{13},\e_{14},\e_{24}\}$ for the distributions associated with $\mathfrak d=\mathfrak p_1\oplus\mathfrak p_3$ (\emph{left}), and $\mathfrak d=\mathfrak p_1\oplus\mathfrak p_{2}\oplus\mathfrak p_3$ (\emph{right}), with the same initial conditions.}
    \label{fig:example2}
\end{figure}

\subsection{Brackets generating distributions of $SO(n)$}

In the two examples above, as well as in the next section, we have presented various examples of $SO(n)$--invariant bracket generating distributions on $SO(n)$. Moreover, since all considered filtrations consist of multiplicity-free, almost multiplicity-free, and symmetric pairs of Lie subalgebras, the corresponding normal geodesic flows are integrable in the classical commutative sense via polynomial first integrals in the momenta~\cite{JSV, Mik}.

However, this represents only a subset of all $SO(n)$--invariant bracket generating distributions.

We have the following statement.

\begin{lemma}\label{lem4}
   The minimal dimension of a completely nonholonomic distribution associated with the Lie algebra $\mathfrak{so}(n)$ is 2.
\end{lemma}

\begin{proof}
   Denote the standard basis of the Lie algebra $\mathfrak{so}(n)$ by $\{\e_{ij}=\e_i\wedge \e_j=-\e_{ji}\mid 1 \leq i < j \leq n\}$. Then the structural equations are given by:
   \begin{align}\label{eq:kom_son}
       [\e_{ij}, \e_{kl}]=\delta_{jk}\e_{il}-\delta_{ik}\e_{jl}+\delta_{il}\e_{jk}-\delta_{jl}\e_{ik},\qquad 1\leq i,j,k,l\leq n.
   \end{align}

   Set:
   \begin{align*}
       v_1 &= \sum_{1<k<n}\e_{k,k+1},& v_2&=\e_{12},\\
       v_3 &=[v_1,v_2]=-\e_{13},&
       v_j&=v_{j-2}+[v_1,v_{j-1}],\quad j=4,\ldots, n.
   \end{align*}
   Then it follows directly that:
   \begin{align*}
       \e_{1j}&= (-1)^j v_j,\quad 1<j\leq n,\\
       \e_{ij}&=[ \e_{1j}, \e_{1i}]=(-1)^{i+j}[v_j, v_i],\quad 1< i < j \leq n.
   \end{align*}

   Hence, $\mathfrak d=\Span\{v_1, v_2\}$ generates completely nonholonomic distribution $\mathcal{D}$.
\end{proof}

\begin{remark}
A more general statement than Lemma~\ref{lem4} holds. Namely, given an arbitrary element $X$ of a simple Lie algebra, there exists a $2$--dimensional completely nonholonomic distribution containing $X$. This follows from the results in \cite{Tud}.
\end{remark}

Consider the left-invariant distribution
$\mathcal D$
related to the distribution
\[
\mathfrak d=\Span\{v_1, v_2\}=\Span\{\e_{23}+\e_{34},\e_{12}\}
\]
on $SO(4)$. We define the sub-Riemannian structure $ds^2_{\mathcal D,\nu}$
by the Hamiltonian
\begin{align*}
H_{sR}(x)=\frac12
\big(\nu_1(x_{23}+x_{34})^2+\nu_2 x_{12}^2\big), \qquad \nu_1,\nu_2>0.
\end{align*}
Then
\begin{align*}
\omega=\nabla H_{sR}=\nu_1(x_{23}+x_{34})(\e_{23}+\e_{34})+\nu_2 x_{12}\e_{12}\in\mathfrak d.
\end{align*}

The normal geodesic flow is given by:
\begin{align}\label{novi0}
\begin{aligned}
    &\dot x=[x,\omega]=[x,\nu_1(x_{23}+x_{34})(\e_{23}+\e_{34})+\nu_2 x_{12}\e_{12}], \\
&\dot R=R\cdot(\nu_1(x_{23}+x_{34})(\e_{23}+\e_{34})+\nu_2 x_{12}\e_{12}), \qquad (R,x)\in SO(4)\times\so(4).
\end{aligned}
\end{align}
In coordinates $x_{ij}$, the Euler equation takes the form:
\begin{align}\label{novi}
\begin{aligned}
&\dot x_{12}=-\nu_1 x_{13}(x_{23}+x_{34}),\\
&\dot x_{13}=\nu_1 (x_{12}-x_{14})(x_{23}+x_{34})-\nu_2 x_{12}x_{23},\\
&\dot x_{14}=\nu_1 x_{13}(x_{23}+x_{34})-\nu_2 x_{12}x_{24},\\
&\dot x_{23}=-\nu_1 x_{24}(x_{23}+x_{34})+\nu_2 x_{12}x_{13},\\
&\dot x_{24}=\nu_1 (x_{23}-x_{34})(x_{23}+x_{34})+\nu_2 x_{12}x_{14},\\
&\dot x_{34}=\nu_1 x_{24}(x_{23}+x_{34}).
\end{aligned}
\end{align}
We have three integrals:
 the Hamiltonian $H_{sR}$ and the Casimirs
\begin{align*}
I_1=x_{12}^2+x_{13}^2+x_{14}^2+x_{23}^2+x_{24}^2+x_{34}^2, \quad
I_2=x_{12}x_{34}-x_{13}x_{24} +x_{14}x_{23}.
\end{align*}
For the complete integrability we need the fourth
independent integral.

\begin{proposition}
The system~\eqref{novi} admits no polynomial integrals of degree up to 6 that are independent from the quadratic integrals: $H_{sR}$, $I_1$, $I_2$.
\end{proposition}
\vspace{-1em}
\begin{proof}
Since the Hamiltonian $H_{sR}$ is homogeneous, any polynomial first integral decomposes into homogeneous components, each of which must itself satisfy the first integral condition; therefore, it suffices to restrict attention to homogeneous polynomials.

Let $p_d(x)$ be a homogeneous polynomial of degree $d$, considered as a candidate for a first integral. From~\eqref{novi}, this leads to a system of equations that is linear in the coefficients of $p_d(x)$. A direct, though somewhat lengthy, computation shows that when $d$ is odd, this system admits only the trivial solution. In contrast, when $d$ is even, any nontrivial solution corresponds to a combination of the quadratic integrals $H_{sR}$, $I_1$, and $I_2$.
The complexity of this system grows exponentially; for instance, when $d=6$, it consists of 755 equations in 462 variables.
\end{proof}

A detail study of the system~\eqref{novi} is out of the scope of the paper. We observe a connection to systems related with chains of subalgebras. The system~\eqref{novi} has the invariant subspace
\begin{align*}
\so(3)=\{x_{12}=0, \, x_{13}=0, \, x_{14}=0\}=\Span\{\e_{23}, \e_{24}, \e_{34}\},
\end{align*}
where it takes the form
\begin{gather}\label{novi1}
\begin{aligned}
\dot x_{23}&=-\nu_1 x_{24}(x_{23}+x_{34}),\\
\dot x_{24}&=\nu_1 (x_{23}-x_{34})(x_{23}+x_{34}),\\
\dot x_{34}&=\nu_1 x_{24}(x_{23}+x_{34}).
\end{aligned}
\end{gather}
The solution curves are the points of intersection between the cylinder $(x_{23}-x_{34})^2+2x_{24}^2=c_1$ and the plane $x_{23}+x_{34}=c_2$.

Together with the system~\eqref{novi1}, let us consider the filtration
\begin{align*}
\g_0=\Span\{\e_{23}+\e_{34}\}<\g_1=\so(3)=\g_0 \oplus \Span\{\e_{23}-\e_{34}, \e_{24}\}<\g_2=\so(4).
\end{align*}

We have
\begin{align*}
&x_0=\frac12(x_{23}+x_{34})(\e_{23}+\e_{34}), \qquad
x_1=\frac12(x_{23}-x_{34})(\e_{23}-\e_{34})+x_{24}\e_{24},\\
&x_2=x_{12}\e_{12}+x_{13}\e_{13}+x_{14}\e_{14},
\end{align*}
and the system~\eqref{Euler-0**},~\eqref{Euler-i**},
\begin{align*}
\dot x_0 = 0,
\qquad \dot x_1 = [(s_1-s_0) x_0,x_1],
\qquad \dot x_2 = [(s_2-s_0) x_0+(s_2-s_0)x_1,x_2],
\end{align*}
restricted to $\so(3)=\{x_2=0\}$ and for $s_1=0$ and $s_0=2\nu_1$,
coincides with~\eqref{novi1}.

Thus, we can use the solutions described in Theorem~\ref{pomocna}:
\begin{align}
\nonumber & x_0(t)=\bar x_0=\bar x_{\g_0}, \\
\label{R(t)} & x_1(t)=\Ad_{\exp(-2\nu_1)\bar x_{\g_0}}(\bar x_1),\\
\nonumber &R(t)=\bar R\cdot\exp(2\nu_1t\bar x_0).
\end{align}
to describe solutions of the normal sub-Riemannian geodesic flow~\eqref{novi0} with the initial conditions $\bar x\in\so(3)$, $R(0)=\bar R$.

\begin{figure}[h]
    \centering
    \includegraphics[width=0.43\linewidth]{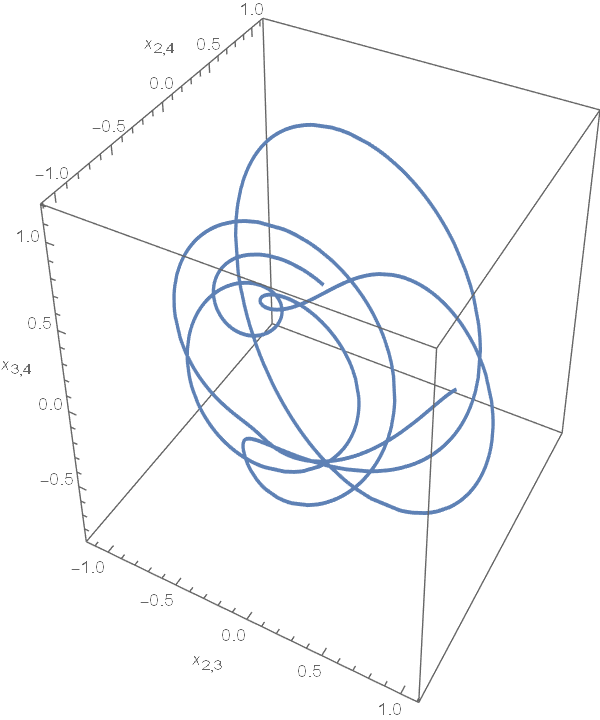}
    \includegraphics[width=0.43\linewidth]{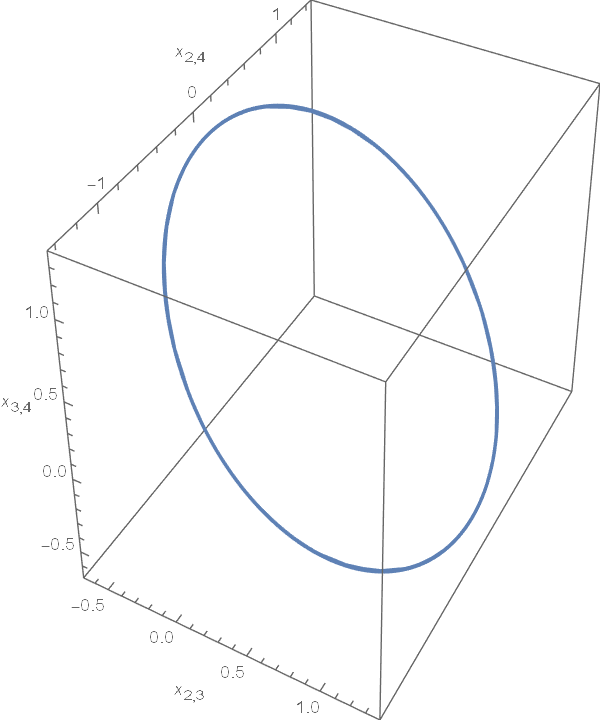}
    \caption{Projection of the integral curve of the system~\eqref{novi} on the subspace $\so(3)=\Span\{\e_{23},\e_{24},\e_{34}\}$ and the level set $I_2=-\frac{1}{2}$ (\emph{left}), and $I_2=0$ (\emph{right}). In the later case, the integral curve belongs to~$\so(3)$.}
    \label{fig:example1}
\end{figure}

We can demonstrate this example by plotting the projection of the integral curve of the system~\eqref{novi} on the subspace $\Span\{\e_{23},\e_{24},\e_{34}\}$. For the metric parameters $\nu_1=2\nu_2=1$ and the level set $H_{sR}=\frac{1}{2}$, we can set the value of the first integral to be $I_1=2$ and examine two distinct cases. The first case corresponds to the initial conditions $x_{12}=1$, $x_{13}=x_{14}=0$, $x_{23}=-x_{34}=\frac{1}{2}$, and $x_{24}=\frac{1}{\sqrt 2}$, which results in a curve on the level set $I_2=-\frac{1}{2}$. The solution curve suggests that the system~\eqref{novi} is not integrable (see Figure~\ref{fig:example1} (\emph{left})). The second case relates to the system~\eqref{novi1}, representing a special case of the integral curves of system~\eqref{novi} on the level set $I_2=0$. The corresponding curve, shown in Figure~\ref{fig:example1} (\emph{right}), is generated from the initial conditions $x_{12}=x_{13}=x_{14}=x_{34}=0$, $x_{23}=\frac{1}{\sqrt 2}$, and $x_{24}=\frac{\sqrt 3}{\sqrt 2}$.

\begin{remark}
In~\cite{BAB}, the following natural question is raised: is  there a left-invariant sub-Riemannian structure of $\rank k$ with an integrable normal geodesic flow on any (semi-simple) Lie group $G$ and any admissible $k$? The example above indicates that for $SO(n)$, the answer is negative.
However, for the orthogonal group and for any $k$, $n-1\leq k <\frac{n(n-1)}{2}$, the answer is positive. For instance, one can consider the natural chains of subalgebras of the form
\begin{align}\label{lanci}
\begin{aligned}
\so(l_1) &< \so(l_1)\oplus \so(l_2) < \so(l_1+l_2) <
\so(l_1+l_2)\oplus \so(l_3)<\so(l_1+l_2+l_3)<\dots\\ & <\so(l_1+\dots+l_{p-1})<\so(l_1+\dots+l_{p-1})\oplus \so(l_p)<\so(n),
\end{aligned}
\end{align}
for all possible decompositions
$n=l_1+l_2+\dots+l_p$, $l_1\ge 2$, $l_i \ge 1$, $1<i\le p$, $1<p<n$, and all possible sets $\mathcal I$ that define $\mathfrak d$ by~\eqref{defd}.
\end{remark}

\begin{remark}
Note that the normal geodesic line $R(t)=\bar R\cdot \exp(2\nu_1 t\bar x_0)$, given in~\eqref{R(t)}, represents a homogeneous geodesic. Specifically, it is the orbit of a one-parameter subgroup of isometries. Homogeneous geodesics in sub-Riemannian manifolds have been recently studied by Podobryaev~\cite{Po}. A sub-Riemannian manifold is called a \emph{geodesic orbit} if every normal geodesic is homogeneous~\cite{Po}.

As an example, consider the sub-Riemannian $\mathfrak p_1\oplus\g_0$--problem on $G$, noted in~\cite{Po} when $(\g,\g_0)$ is a symmetric pair. In this case, $G$ can be viewed as the homogeneous space $(G\times G_0)/\Delta G_0$, where $\Delta G_0=\{(h,h)\,\vert\, h\in G_0\}<G\times G_0$.

In every coset $(g,h)\Delta G_0$ we have a unique representative $(gh^{-1},e)$ with the second factor equal to the neutral in $G_0$.  Thus, the projection $\pi\colon G\times G_0\to (G\times G_0)/\Delta G_0\cong G$, can be represent as $\pi(g,h)=gh^{-1}$. Let $\bar\ell_{(a,b)}\colon G\times G_0\to G\times G_0$, $(a,b)\in (G,G_0)$ be the standard left-action: $\bar\ell_{(a,b)}((g,h))=(ag,bh)$, and let $\ell_{(a,b)}\colon G\to G$ be the action
\begin{align*}
\ell_{(a,b)}(g)=agb^{-1}, \qquad (a,b)\in G\times G_0.
\end{align*}
Then the actions commute with the projection $\pi$:
$\pi\circ\bar \ell_{(a,b)}=\ell_{(a,b)}\circ\pi$.

The action $\ell$ is the action by isometries
with respect to $\mathfrak p_1\oplus\g_0$--sub-Riemannian structure on $G$.
Therefore, the Agrachev--Brockett--Jurjdevic solution~\eqref{rekonstrukcija} represents the action
of one-parametric subgroup $(\exp(ts_1 \bar x),\exp(ts_1\bar x_0))$ within group of isometries $G\times G_0$.
\end{remark}

\section{Relation with sub-Riemannian Manakov's metrics}\label{sec5}

In this section we compare the above construction with the sub-Riemannian structures obtained from the Manakov metrics on the orthogonal group $SO(n)$ and a
class of homogeneous spaces of $SO(n)$.

\subsection{Sub-Riemannian Manakov metrics on $SO(n)$}
Let $\mathbf a$ and $\mathbf b$ be diagonal matrices $\mathbf a=\diag(a_1,\dots,a_n)$,
$\mathbf b=\diag(b_1,\dots,b_n)$ with different eigenvalues. Consider the operator:
\begin{align}\label{manakov}
A=\ad_{\mathbf a}^{-1}\circ\ad_{\mathbf b}=
\ad_{\mathbf b}\circ\ad_{\mathbf a}^{-1}\colon \mathfrak{so}(n)\longrightarrow \mathfrak{so}(n) \Longleftrightarrow \omega_{ij}=\frac{b_i-b_j}{a_i-a_j}x_{ij}, \, 1\le i<j\le n.
\end{align}

In the case $(b_i-b_j)/(a_i-a_j)>0$ for all $i\ne j$,
the inverse operator $I=A^{-1}=
\ad_{\mathbf a}\circ\ad_{\mathbf b}^{-1}$ defines the left-invariant Riemannian metric on $SO(n)$ by the scalar product
\begin{align*}
(\xi,\eta)_I=\langle I(\xi),\eta\rangle=\sum_{i<j}\frac{a_i-a_j}{b_i-b_j}\xi_{ij}\eta_{ij}.
\end{align*}

Manakov obtained polynomial integrals
\begin{align}\label{manakov-integrals}
\mathcal L=\{\tr(x+\lambda\mathbf a)^k\,\vert\, k=1,\dots,n, \, \lambda\in\R\}
\end{align}
and solved the corresponding Euler equation
\begin{align}\label{euler-manakov}
\dot x=[x,\omega], \qquad \omega=A(x)=\ad_{\mathbf a}^{-1}\circ\ad_{\mathbf b}(x)
\end{align}
in terms of theta functions~\cite{Ma}.
Mishchenko and Fomenko proved that the Manakov integrals
\eqref{manakov-integrals} form the complete
commutative set on $\so(n)^*$ (see~\cite{MF1}),
implying that the geodesic flow on $T^*SO(n)$
with the additional kinematic equation
\begin{align}\label{kin-manakov}
\dot R=R\cdot \omega, \qquad R\in SO(n)
\end{align}
is completely integrable in the non-commutative sense~\cite{MF}.

In~\cite{DGJ2009} it is noted that we can
consider the Manakov operator~\eqref{manakov} when
some of the parameters $b_i$ are equal and that then the system~\eqref{euler-manakov},~\eqref{kin-manakov}  represents the sub-Riemannian geodesic flow on $T^*SO(n)$. More precisely, suppose that
$b_1=\dots=b_{l_1}=\beta_1,\dots,
b_{n+1-l_p}=\dots=b_n=\beta_{p}$,
$l_1+l_2+\dots+l_{p}=n$, $\beta_i\ne \beta_j$, $i\ne j$, and that $\ad_\mathbf b\circ\ad_\mathbf a^{-1}$ is a positive definite restricted to $\mathfrak d$.
Here $\mathfrak d$ is the orthogonal complement of
the isotropy subalgebra
\begin{align*}
\so(n)_\mathbf b=\{\xi\in \so(n)\,
\vert\, [\xi,\mathbf b]=0\}=
\so(l_1)\oplus
\so(l_2)\oplus\dots\oplus \so(l_p)
\end{align*}
 with respect to the invariant scalar product
$\langle\cdot,\cdot\rangle$: $\so(n)=\so(n)_\mathbf b\oplus\mathfrak d$. Note that $\ad_\mathfrak b^{-1}$
is well defined on $\mathfrak d$ and that
$(\ad_\mathbf b\circ\ad_\mathbf a^{-1}\vert_\mathfrak d)^{-1}=\ad_\mathbf a\circ\ad_\mathbf b^{-1}\vert_\mathfrak d$.

The linear subspace $\mathfrak d$ always generates
$\so(n)$. Thus, the Manakov operator
\eqref{manakov} defines the sub-Riemannian
structure $ds^2_{\mathcal D,\mathbf a,\mathbf b}$ on the left-invariant distribution $\mathcal D$ (see~\eqref{mathcal D}) on $SO(n)$ by
the scalar product:
\begin{align}\label{manakov metric}
(\xi,\eta)_\mathfrak d=\langle
\ad_{\mathbf b}^{-1}\circ\ad_{\mathbf a}(\xi),\eta\rangle, \qquad \xi,\eta\in\mathfrak d.
\end{align}

Thus, according to~\cite{Ma, MF1, MF}  we have the following statement.

\begin{theorem}\label{man1}
Assume that all $a_i$ are mutually different.
The Euler equation~\eqref{euler-manakov} of the normal geodesic flow with the left-invariant sub-Riemannian structure $ds^2_{\mathcal D,\mathbf a,\mathbf b}$  is completely integrable by means of commuting integrals~\eqref{manakov-integrals}.
The normal geodesic flow~\eqref{euler-manakov},~\eqref{kin-manakov} on the phase space $T^*SO(n)$ is completely integrable in the non-commutative sense by means of the Manakov integrals~\eqref{manakov-integrals} and the components of momentum map
$\Phi(R,x)=\Ad_R(x)$. Generic motions are quasi-periodic winding over invariant isotropic tori of dimension:
\begin{align*}
\Delta=\frac12\big(\dim\so(n)+\rank\so(n)\big).
\end{align*}
\end{theorem}

\begin{example}\label{bloch}
Let  $\mathbf b=(b_1,b_2\dots,b_2)$, $b_1> b_2$, and $a_1>a_i$, $a_i\ne a_j$, $i,j>1$, $i\ne j$. Then
\begin{align*}
\mathfrak d=\Span\{ \mathbf e_{1i}=\mathbf e_1\wedge \mathbf e_i\, \vert\, i=2,\dots,n\},
\end{align*}
and the sub-Riemannian structure
is given by
\begin{align*}
(\xi,\eta)_\mathfrak d=\sum_{i=2}^m A_i\xi_{1i}\eta_{1j}, \qquad A_i=(a_1-a_i)/(b_1-b_i), \qquad i=2,\dots,n.
\end{align*}
The corresponding normal
sub-Riemanian geodesic flow, in the right-invariant formulation, was thoroughly examined in~\cite{BAB}. It is related to an optimal problem of a rubber ball rolling over a hyperplane.
\end{example}

The Manakov sub-Riemannain metric
given in Example~\ref{bloch} is well defined
if some parameters $a_i$ are mutually equal.
In particular, it fits into the chain of subalgebras construction
for $a_1\ne a_2=\dots=a_n$.
Thus, it is interesting to consider matrixes $\mathbf a$ with multiple eigenvalues. In~\cite{DGJ2009}, the authors refer to the corresponding systems as singular Manakov flows.

The construction used in~\cite{DGJ2009} can be easily adapted to sub-Riemannian structures. Suppose that
$a_1=\dots=a_{k_1}=\alpha_1,\dots,a_{n+1-k_r}=\dots=a_n=\alpha_r$, $k_1+k_2+\dots+k_r=n$,
$\alpha_i\ne\alpha_j$, $i\ne j$, such that
\begin{align*}
\so(n)_\mathbf a=\{\xi\in \so(n)\,
\vert\, [\xi,\mathbf a]=0\}=\so(k_1)\oplus\dots\oplus\so(k_r)\le\so(n)_\mathbf b,
\end{align*}

Let $\mathfrak v$ be the orthogonal complement of
$\so(n)_\mathbf a$: $\so(n)=\so(n)_\mathbf a\oplus \mathfrak v$ (note that $\mathfrak d \le \mathfrak v$).
Now $\ad^{-1}_\mathbf a$ is well defined on $\mathfrak v$ and we assume that $\ad_\mathbf b\circ\ad_\mathbf a^{-1}$ is a positive definite operator restricted to $\mathfrak d$. Again, we have the sub-Riemannian structure
$ds^2_{\mathcal D,\mathbf a,\mathbf b}$
on the left-invariant distribution $\mathcal D$ defined by the scalar product~\eqref{manakov metric}.
The Hamiltonian function of the normal geodesic flow
reads
\begin{align}\label{sub-manakov}
H_{sR,\mathbf a,\mathbf b}(x)=\frac12\langle \ad_\mathbf b\circ\ad_\mathbf a^{-1}(x_\mathfrak v),x\rangle
\end{align}
where by $x_\mathfrak v$ we denote the orthogonal projection of $x\in\so(n)^*\cong \so(n)$ to $\mathfrak v$ with respect to $\langle\cdot,\cdot\rangle$.

From the relations
\begin{align*}
\pr_{\so(n)_\mathbf a}[x_\mathfrak v,\ad_\mathbf a^{-1}\ad_\mathbf b (x_\mathfrak
v)]=0, \qquad [\so(n)_\mathbf a,\mathfrak v]\subset\mathfrak v
\end{align*}
(see~\cite{DGJ2009} with $A=\mathbf a$, $B=\mathbf b$, $M=x$, $\Omega=\omega$, $\mathfrak B=0$)\footnote{There is typo in~\cite[eq. (14)]{DGJ2009}, where one $M_{\mathfrak v}$ is missing.},
the Euler equation now takes the form of the singular Manakov flow:
\begin{align}
&\dot x_{\so(n)_\mathbf a}=0, \label{s0} \\
&\dot x_\mathfrak v=[x_{\so(n)_\mathbf a}+x_{\mathfrak v},\ad_\mathbf a^{-1}\ad_\mathbf b (x_\mathfrak v)].
\label{s1}
\end{align}

From~\cite[Theorem 1]{DGJ2009} we get:

\begin{theorem}\label{man2}
The normal geodesic flow~\eqref{s0},~\eqref{s1},~\eqref{kin-manakov} of the left-invariant sub-Riema\-nnian structure $ds^2_{\mathcal D,\mathbf a,\mathbf b}$ is completely integrable in the non-commutative sense by means of the Manakov integrals~\eqref{manakov-integrals}, the components of preserved angular momentum $x_{\so(n)_\mathbf a}$, and the components of momentum map $\Phi(R,x)=\Ad_R(x)$.
\end{theorem}

On the other hand, on $\mathcal D$ we can define sub-Riemannian structures by using chains of Lie subalgebras. For example, we can take the chain~\eqref{lanci} with a suitable choice of the set $\mathcal I$. Therefore, for the left-invariant bracket generating distributions $\mathcal D$ induced from $\mathfrak d=(so(n)_\mathbf b)^\perp$, we have two natural constructions of sub-Rieammanian structures with completely integrable geodesic flows: by using the chains of Lie subalgebras and the Manakov sub-Riemannian metrics. However, the structures are different in general.

\begin{example}\label{M1primer}
Assume $b_1=\dots=b_{l_1}=\beta_1$, $b_{l_1+1}=\dots=b_n=\beta_2$,
\begin{align*}
\mathfrak d=\big(\so(l_1)\oplus\so(l_2)\big)^\perp
=\Span\{\mathbf e_{ij}\,\vert\, 1\le i \le l_1, \, l_1+1\le j \le n\}.
\end{align*}
Note that $(\so(n),\so(l_1)\otimes \so(l_2))$ is a symmetric pair.

  The condition $\so(n)_\mathbf a\le \so(n)_\mathbf b$ implies that we can have equalities only between parameters $a_i$ with indexes that belong to disjoint sets $\{1,2,\dots,l_1\}$ and $\{l_1+1, l_1+2,\dots,n\}$.
The Manakov metrics $ds^2_{\mathcal D,\mathbf a,\mathbf b}$ are
defined by the scalar product
\begin{align*}
(\xi,\eta)_\mathfrak d=\sum_{i=1}^{l_1}\sum_{j=l_1+1}^n \frac{a_i-a_j}{\beta_1-\beta_2}\xi_{ij}\eta_{ij}, \qquad \xi,\eta\in\mathfrak d.
\end{align*}

When $l_1=1$, we have the case considered in Example~\ref{bloch}. For $a_1=\dots=a_{l_1}=\alpha_1$, $a_{l_1+1}=\dots=a_n=\alpha_2$, the sub-Riemmanian Manakov structure coincides with the structure related to the chain
\begin{align*}
\g_0=\so(l_1)\oplus\so(l_2)<\g_1=\so(n)
\end{align*}
with $s_1=(\beta_1-\beta_2)/(\alpha_1-\alpha_2)$. This is illustated in Figure~\ref{fig:example3} (\emph{left}). Note that this is the same solution curve as the curve from Subsection~\ref{ssec:42} with $s_0=s_2=0$ and $s_1=s_3=(\beta_1-\beta_2)/(\alpha_1-\alpha_2)$.

\begin{figure}[h]
    \centering
    \includegraphics[width=0.43\linewidth]{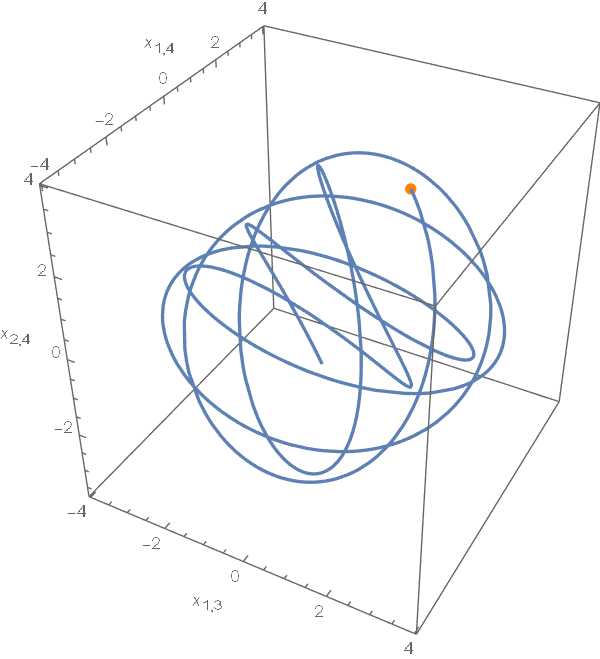}
    \includegraphics[width=0.43\linewidth]{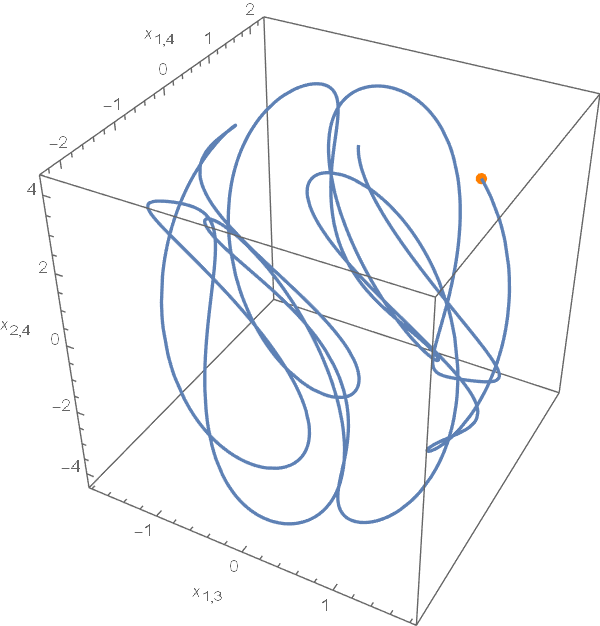}
    \caption{Integral curves for the pair $(\so(4),\so(2)\otimes \so(2))$ and Manakov metrics with $a_1\!=\!a_2$, $a_3\!=\!a_4$ (\emph{left}), and $a_1\!\neq\! a_2$, $a_3\!\neq\! a_4$ (\emph{right}).}
    \label{fig:example3}
\end{figure}
\end{example}

\begin{example}\label{M2primer}
Let $b_1=\dots=b_{l_1}=\beta_1$,
$b_{l_1+1}=\dots=b_{l_1+l_2}=\beta_2$, $b_{l_1+l_2+1}=\dots=b_n=\beta_3$. Then
\begin{align*}
\mathfrak d=\big(\so(l_1)\oplus\so(l_2)\oplus\so(l_3)\big)^\perp, \qquad l_1+l_2+l_3=n.
\end{align*}
For
$a_1=\dots=a_{l_1}=\alpha_1$, $a_{l_1+1}=\dots=a_{l_2}=\alpha_2$, $a_{l_1+l_2+1}=\dots=a_n=\alpha_3$,
such that
\begin{align*}
\alpha_1-\alpha_2=\alpha_2-\alpha_3, \quad \beta_1-\beta_2=\beta_2-\beta_3,
\end{align*}
the sub-Riemmanian Manakov structure coincides with the structure related to the chain
\begin{align*}
\g_0=\so(l_1)\oplus\so(l_2)\oplus\so(l_3)<\g_1=\so(n)
\end{align*}
with $s_1=(\beta_1-\beta_2)/(\alpha_1-\alpha_2)=
(\beta_2-\beta_3)/(\alpha_2-\alpha_3)=(\beta_1-\beta_3)/(\alpha_1-\alpha_3)$.
\end{example}

\subsection{Sub-Riemannian Manakov metrics on homogeneous spaces of $SO(n)$}

Let $ds^2_{\mathcal D,\mathbf a,\mathbf b}$ be the sub-Riemannian structure on $SO(n)$ with the matrix $\mathbf a$ with multiple eigenvalues.
Let $K$ be a subgroup of $SO(n)$ with the Lie algebra
$\mathfrak k=Lie(K)$, which is a Lie subalgebra of
$\so(n)_\mathbf a$ ($\mathfrak k\le \so(n)_\mathbf a$). We additionally suppose that
\begin{align*}
\mathfrak k\oplus\mathfrak d\ne \so(n) \quad \Longleftrightarrow \quad \mathfrak d<\mathfrak k^\perp.
\end{align*}

Repeating the construction from Section~\ref{redukcija}, we consider the right $K$--action on $SO(n)$ and $T^*SO(n)$, homogeneous space $SO(n)/K$ and its cotangent bundle
$T^*(SO(n)/K)\cong (T^*SO(n))_0/K$,
\begin{align*}
(T^*SO(n))_0=\{(R,x)\in SO(n)\times \so(n)^*\cong T^*SO(n), \vert\, \pr_{\mathfrak k}(x)=0\}.
\end{align*}
Since $\mathfrak d<\mathfrak k^\perp$, the linear space $\mathfrak d$ defines the $SO(n)$--invariant bracket generating distribution $\mathcal D_0$ on $SO(n)/K$. We obtain the sub-Riemannian structure $ds^2_{\mathcal D_0,\mathbf a,\mathbf b}$ on $\mathcal D_0$ with the Hamiltonian function $H_{sR,0}$ induced from the restriction of the Hamiltonian~\eqref{sub-manakov} to $(T^*SO(n))_0$.

All integrals mentioned in Theorem~\ref{man2} are right
$K$--invariant, and their restrictions to $(T^*SO(n))_0$ project to the cotangent bundle
$T^*(SO(n)/K)$. The completeness of these integrals is proven in~\cite{DGJ2009, DGJ2015, Mik2}.
We note that proof of this statement does not follow from Theorem~\ref{man2} and requires additional techniques. Thus, we get

\begin{theorem}\label{man3}
The normal sub-Riemannian geodesic flow of the Manakov sub-Riemannian structure $ds^2_{\mathcal D_0,\mathbf a,\mathbf b}$ on the homogeneous space $SO(n)/K$ is completely integrable in the non-commutative sense.
The complete set of integrals on $T^*(SO(n)/K)$
is induced from restrictions of the Manakov integrals~\eqref{manakov-integrals}, the components of angular momentum $x_{\so(n)_\mathbf a}$, and the components of momentum map $\Phi(R,x)=\Ad_R(x)$ to
$(T^*SO(n))_0$.
\end{theorem}

\begin{example}
Let us consider Example~\ref{bloch} and
\begin{align*}
K=\{R\in SO(n)\, \vert \, R=\diag(1,1,S), \, S\in SO(n-2)\}.
\end{align*}
Then $SO(n)/K=SO(n)/SO(n-2)$ is the rank two Stiefel
variety $V_{n,2}$, and $\mathfrak d$ is a subspace of
$\mathfrak so(n-2)^\perp$ of codimension $n-2$.
The distribution $\mathcal D_0\subset TV_{n,2}$ coincides with the distribution $\mathcal D_0$
considered in~\cite{Jo2025}. For $a_1\ne a_2=\dots=a_n$, and the chain of subalgebras
\begin{align*}
\mathfrak k=\g_0=\so(n-2)<\g_1=\so(n-1)<\g_2=\so(n),
\end{align*}
we have the example for Theorem 2, with
$s_0=s_1=0$, $s_2=(b_1-b_2)/(a_1-a_2)$.

Now, consider Example~\ref{M1primer} with
$b_1=b_2\ne b_3=\dots=b_n$ and the same subgroup $K=SO(n-2)$. Then $\mathfrak d$ is a subspace of codimension 1 in $\mathfrak so(n-2)^\perp$. The distribution $\mathcal D_0$ coincides with the contact distribution $\mathcal H\subset TV_{n,2}$ considered in~\cite{Jo2025}.
For $a_1=a_2\ne a_3=\dots=a_n$
and the chain of subalgebras
\begin{align*}
\mathfrak k=\g_0=\so(n-2)<\g_1=\so(n-2)\oplus so(2)<\g_2=\so(n),
\end{align*}
we also have the example for Theorem 2, with
$s_0=s_1=0$, $s_2=(b_1-b_2)/(a_1-a_2)$.

The above distributions are, up to conjugation,
all $SO(n)$--invariant, bracket generating distribution of $TV_{n,2}$.
On the other hand,  $ds^2_{\mathcal D_0,s}$ and $ds^2_{\mathcal D_0,\mathbf a,\mathbf b}$ do not cover all possible $SO(n)$--invariant sub-Riemannian structures on $V_{n,r}$. A detailed description can be found in~\cite{Jo2025}.
\end{example}

\subsection*{Acknowledgements}
We thank the referee for careful reading and useful remarks. This research is part of the project IntegraRS of the Science Fund of Serbia. The research of B.J. was supported by the Serbian Ministry of Education, Science and
Technological Development through Mathematical Institute of Serbian
Academy of Sciences and Arts.
The research of T.\v S. and S.V. is partially supported by the Ministry of  Science, Technological Development and Innovation, Republic of Serbia, through the project 451-03-33/2026-03/200104.


\begin{thebibliography}{99}

\bibitem{Ag}
Agrachev, A. A.: \emph{Methods of control theory in nonholonomic geometry}.
In Proceedings of the International
Congress of Mathematicians, Vols. 1, 2, pp. 1473--1483, Basel, Birkh\"{a}user, 1995.

\bibitem{ABB}
Agrachev, A,.  Barilari, D.,  Boscain, U., \textit{A comprehensive
introduction to sub-Riemannian geometry},  Cambridge University
Press, 2019.

\bibitem{Ar}  Arnol'd, V. I., \textit{Mathematical methods of classical
    mechanics}, 2nd ed., Springer, 1989.



\bibitem{B1} Bogoyavlenski, O.I.: \emph{Integrable Euler equations
associated with filtrations of Lie algebras},  Math.
USSR Sb. {\bf 49} (1984), Issue 1, 229--238.

\bibitem{BJ1}
Bolsinov, A. V. and Jovanovi\' c, B.: \emph{Integrable geodesic flows on homogeneous spaces},  Sb. Math., \textbf{192} (2001), Issue 7, 951--968.

\bibitem{BJ3} Bolsinov, A. V. and Jovanovi\' c, B.:
\emph{Complete involutive algebras of functions on cotangent bundles of
homogeneous spaces}, Mathematische Zeitschrift {\bf 246} (2004) no.
1-2, 213--236.

\bibitem{BMZ}
Bouloc D, Miranda E, Zung
N. T. 2018 Singular fibres of the Gelfand-Cetlin
system on $u(n)^*$. Phil. Trans. R. Soc. A 376:
20170423.

\bibitem{BAB} Bravo-Doddoli, A.,  Arathoon, P.,  Bloch, A. M.,  Integrable sub-Riemannian geodesic flows on the special orthogonal group, arXiv:2411.09008 [math.DG]

\bibitem{Br}
Brockett, R.W.: \emph{Explicitly solvable control problems with nonholonomic constraints}. In Decision and control,
1999, Proc. 38th IEEE Conf. on, vol. 1, pp. 13--16, IEEE, 1999.

\bibitem{DGJ2009} Dragovi\' c, V., Gaji\' c, B., Jovanovi\' c, B., Singular Manakov flows and geodesic flows on homogeneous spaces of $SO(N)$, \textit{Transform. Groups}, 2009, vol. {14}, pp. 513--530.


\bibitem{DGJ2015} B. Gaji\' c, V. Dragovi\' c, B. Jovanovi\' c, \emph{On the completeness of Manakov integrals}. Fundam. Prikl.
Mat., \textbf{20} (2015) Issue 2, 35–-49 (Russian). English translation: J. Math. Sci.,  \textbf{223}(2017) Issue 6, 675-–685,   	arXiv:1504.07221 [nlin.SI]


\bibitem{GS1} Guillemin, V and Sternberg, S.:
\emph{On collective complete integrability according to the method of
Thimm}, Ergod. Th. \& Dynam. Sys.  {\bf 3} (1983)  219-230.

\bibitem{Tud} Ionescu, T.: \emph{On the generators of semisimple lie algebras}, Linear Algebra Appl. \textbf{15}(1976), Issue 3, 271--292.


\bibitem{Jo} Jovanovi\' c, B.:
Geometry and integrability of Euler-Poincar\'e-Suslov
equations, Nonlinearity, \textbf{14}  (2001) 1555--1567.


\bibitem{JSV} Jovanovi\' c, B, \v Sukilovi\' c, T., and Vukmirovi\' c, S.: \emph{Integrable systems associated to the filtrations of Lie algebras}, Regul. Chaot. Dyn. \textbf{28} (2023), 44--61.

\bibitem{JSV2024} Jovanovi\' c, B. \v Sukilovi\' c, T., Vukmirovi\' c, S.,
Almost multiplicity free subgroups of compact Lie groups and polynomial integrability of sub-Riemannian geodesic flows,
\textit{Letters in Mathematical Physics}, 2024, vol. {114},  14, 16 p.

\bibitem{Jo2025}  Jovanovi\' c, B.: Contact magnetic geodesic and sub-Riemannian flows on $V_{n,2}$ and integrable cases of a heavy rigid body with a gyrostat,
Regul. Chaot. Dyn. \textbf{30} (2025)
arXiv:2506.13101 [math.DG].

\bibitem{Ju1999}
Jurdjevic, V.: Optimal Control, Geometry and Mechanics. Mathematical Control Theory., edited by J. Bailleu and J.C. Willems, Springer (1999), pp. 227--267.

\bibitem{Ju} Jurdjevic, V.:  \textit{Optimal control and geometry: integrable systems}, Cambridge University Press,
Cambridge, 2016.

\bibitem{Ju2} Jurdjevic, V.: Rolling geodesics on symmetric
semi-riemannian spaces, Theoretical and Applied Mechanics, (2025),
https://doi.org/10.2298/TAM250408015J.

\bibitem{Ma}
Manakov, S. V., Note on the integration of Euler's equations
of the dynamics of an $n$--dimensional rigid body,  \textit{Funct. Anal. Appl.}, 1977, vol. {10}, no. 4, pp. 328--329.

\bibitem{MMR} Marsden, J. E., R. Montgomery, and T. S. Ratiu: \textit{Reduction, Symmetry and Phases in Mechanics}, Memoirs of the American Mathematical Society \textbf{88(436)} 1990.

\bibitem{Mik} Mikityuk, I. V.:
\emph{Integrability of the Euler equations associated with filtrations
of semisimple Lie algebras}, {Math. USSR Sbornik} {\bf 53} (1986)  541--549.

\bibitem{Mik2} I. V. Mykytyuk,  \emph{Integrability of geodesic flows for metrics on suborbits of the adjoint
orbits of compact groups}, Transformation Groups \textbf{21} (2016) 531-–553,  arXiv:1402.6526.

\bibitem{MF1}  Mishchenko, A. S. and Fomenko, A. T.: Euler equations on finite-dimensional
Lie groups, Math. USSR-Izv. 12 (1978), no. 2, 371--389.


\bibitem{MF}  Mishchenko, A. S. and Fomenko, A. T.:
\emph{Generalized Liouville method of integration of Hamiltonian
systems}.  Funct. Anal. Appl. {\bf 12}  (1978)
113--121.

\bibitem{Mo} Montgomery, R.:  \emph{A tour of subriemannnian geometries, their geodesics and applications}, {Amer. Math. Soc.}, 2002.


\bibitem{N}
Nekhoroshev, N. N., Action-angle variables and their generalization, \textit{Tr. Mosk. Mat. O.-va.}, 1972, vol. {26}, pp. 181--198 (Russian).
English translation: \textit{Trans. Mosc. Math. Soc.}, 1972, vol. 26, pp. 180--198.

\bibitem{PS}
Pavlovi\'c, M., \v Sukilovi\' c, T.,
Integrability of the sub-Riemannian geodesic flow of the left-invariant metric on the Heisenberg group,  	arXiv:2404.06586 [math.DG].

\bibitem{Po}
Podobryaev, A. V.,
\emph{Homogeneous geodesics in sub-Riemannian geometry},  	
ESAIM: Control, Optimisation and Calculus of Variations. 29, 11 (2023),  	arXiv:2202.09085 [math.DG]

\bibitem{Sac}
Sachkov, Yu. L.: \emph{Left-invariant optimal control problems on Lie groups: classification and problems integrable by elementary functions},
Russian Math. Surveys, \textbf{77} (2022), 99--163.

\bibitem{Su}
Souris, N. P.: \emph{Geodesics as products of one-parameter subgroups in
compact lie groups and homogeneous spaces}, Mathematische Nachrichten: Volume 296, Issue 6
Pages: 2609--2625.
June 2023

\bibitem{Th} Thimm A.:
\emph{Integrable geodesic flows on homogeneous spaces}, Ergod. Th. \&
Dynam. Sys.,{\bf 1} (1981) 495--517.


\end{thebibliography}
\end{document}